\documentclass[11pt]{article}

\usepackage{amsmath,amsfonts,amsthm,amssymb,graphicx,tikz,epsfig}
\usepackage[all,2cell,ps]{xy}
\usepackage[T1]{fontenc}

\allowdisplaybreaks

\theoremstyle{plain}
\newtheorem{Theorem}{Theorem}[section]
\newtheorem{Lemma}[Theorem]{Lemma}
\newtheorem{Proposition}[Theorem]{Proposition}
\newtheorem{Corollary}[Theorem]{Corollary}

\theoremstyle{definition}

\newcommand{\R}{\mathbb{R}}

\DeclareMathOperator{\dnu}{\mathrm{d}\nu(x)}

\title{Explicit bounds from the Alon--Boppana theorem}

\author{Joseph Richey\\ \texttt{josephlr@umich.edu}\\ University of Michigan \and Noah Shutty\footnote{Partially supported by the University of Michigan Undergraduate Research Opportunities Program.}\\ \texttt{noajshu@umich.edu}\\ University of Michigan \and Matthew Stover\footnote{This material is based upon work supported by the National Science Foundation under Grant Numbers DMS 1045119 and 1361000. The author acknowledges support from U.S. National Science Foundation grants DMS 1107452, 1107263, 1107367 "RNMS: GEometric structures And Representation varieties" (the GEAR Network).}\\ \texttt{mstover@temple.edu}\\ Temple University}

\date{\today}

\begin{document}

\maketitle

\begin{abstract}
The purpose of this paper is to give explicit methods for bounding the number of vertices of finite $k$-regular graphs with given second eigenvalue. Let $X$ be a finite $k$-regular graph and $\mu_1(X)$ the second largest eigenvalue of its adjacency matrix. It follows from the well-known Alon--Boppana Theorem, that for any $\epsilon > 0$ there are only finitely many such $X$ with $\mu_1(X) < (2 - \epsilon) \sqrt{k - 1}$, and we effectively implement Serre's quantitative version of this result. For any $k$ and $\epsilon$, this gives an explicit upper bound on the number of vertices in a $k$-regular graph with $\mu_1(X) < (2 - \epsilon) \sqrt{k - 1}$.
\end{abstract}

\section{Introduction}\label{sec:intro}

The purpose of this paper is to give explicit methods for bounding asymptotic behavior of the spectrum of finite $k$-regular graphs. We begin with notation that will be used throughout. Fix $k \ge 3$ and let $X$ be a finite connected $k$-regular graph with $n$ vertices. If $A_X$ is the adjacency matrix of $X$, let
\[
\mu_0(X) \ge \mu_1(X) \ge \cdots \ge \mu_{n - 1}(X)
\]
be its eigenvalues, i.e., its \emph{spectrum}. It is well-known that $\mu_0(X) = k$, and $|\mu_j(X)| \le k$ for all $0 \le j \le n - 1$. Throughout this paper we study $\mu_j$ and leave it to the interested reader to convert our results to the \emph{Laplacian spectrum} of $X$
\[
\{ \lambda_j(X) = k - \mu_j(X) \} \subset [0, 2 k].
\]
Our starting point is the following quantitative version, due to Serre, of the famous theorem of Alon and Boppana (see \cite{Alon--Boppana}, \cite{HLW}, or \cite[Theorem 1.4.9]{DSV}).

\begin{Theorem}[Quantitative Alon--Boppana Theorem]\label{thm:QABT}
For any $\epsilon > 0$ and natural number $k$, there exists a positive real constant $C(k, \epsilon)$ such that for any $k$-regular graph $X$ on $n$ vertices
\begin{equation}\label{ABT}
\# \left\{ j\ \mid\ \mu_j(X) \in [(2 - \epsilon) \sqrt{k - 1}, k] \right\} \ge C(k, \epsilon) n.
\end{equation}
\end{Theorem}

See \cite{Cioaba} for another elementary proof, and see \cite{Mohar} for a proof that also includes multipartite graphs and an improvement on Theorem \ref{thm:QABT} for graphs of bounded global girth (the original extension to irregular graphs was due to Y.\ Greenberg's Ph.D.\ thesis; see \cite{Cioaba2}). See \cite{Hoory} for a further generalization.

Though our methods can be used to study the entire spectrum, the focus of this paper will be the behavior of arguably the most important piece of the spectrum, $\mu_1(X)$. In particular, we study the following well-known variant of Theorem \ref{thm:QABT}.

\begin{Theorem}[Finiteness for small $\mu_1$]\label{thm:Finiteness}
For any integer $k \ge 3$ and real number $z < 2 \sqrt{k - 1}$, there are only finitely many $k$-regular graphs $X$ with $\mu_1(X) = z$.
\end{Theorem}

Let $v(k, z) : \mathbb{N} \times \R \to \mathbb{N} \cup \{\infty\}$ be the maximum number of vertices of a $k$-regular graph $X$ with $\mu_1(X) \le z$. Then $v(k, z) < \infty$ for $z < 2 \sqrt{k - 1}$. Notice that the existence of $k$-regular Ramanujan graphs, only very recently proven for all $k$ in \cite{Marcus--Spielman--Srivastava}, implies that there is an infinite sequence $\{X_i\}$ of $k$-regular graphs with $\mu_1(X_i) \leq 2 \sqrt{k - 1}$ for all $i$. Thus $v(k, z) = \infty$ for $z \ge 2 \sqrt{k - 1}$.

In fact, one can calculate $c$ explicitly from Theorem \ref{thm:QABT} by taking any $v(k, z) < C(k, \epsilon)^{-1}$, where $z = (2 - \epsilon) \sqrt{k - 1}$. In \S\ref{sec:ProofOfABT} of this paper, we sketch a proof of Theorem \ref{thm:QABT} and describe how one can extract explicit bounds for $v(k, z)$. Our methods lead to the following theorem.

\begin{Theorem}\label{thm:IntroMachine}
Let
$k \ge 3$ be an integer, $z$ be any real number such that $z < 2 \sqrt{k - 1}$, and $m$ be the smallest integer such that
\[
\frac{z}{\sqrt{k - 1}} < \alpha_m = 2 \cos(\pi / m + 1).
\]
Let $U_j$ be the $j^{th}$ Chebyshev polynomial of second kind, set $V_j(x) = U_j(\frac{x}{2})$, and define functions:
\begin{eqnarray}
F_m(x) &=& \sum_{j = 0}^m V_{2 j}(x) \nonumber \\
\widehat{F}_m(x) &=& \frac{F_m(x)}{x - \alpha_m} \nonumber
\end{eqnarray}
For any $0 < s < \alpha_m - \frac{z}{\sqrt{k - 1}}$, write the function $x \mapsto \widehat{F}_m(x + s)$ in the form $\sum c_j(s) V_j(x)$. Then $c_j(s) \ge 0$ for all $j$, $c_0(s) > 0$, and
\begin{equation}\label{eq:IntroCBound}
v(k, z) \le \frac{\widehat{F}_m(L + s)}{c_0(s)}.
\end{equation}
\end{Theorem}

We prove Theorem \ref{thm:IntroMachine} in \S \ref{sec:Asymptotic} using a general version of \eqref{eq:IntroCBound} and close consideration of the functions $\widehat{F}_m$. With notation as in Theorem \ref{thm:IntroMachine}, note that this gives a bound for $v(k, z)$ that is of order $k^{\frac{2 m - 1}{2}}$. According to the excellent survey of Hoory, Linial, and Wigderson \cite{HLW}, if $z = 2 \sqrt{k - 1} - \epsilon^\prime$, then
\[
v(k, z) = \mathrm{O} \left( (k - 1)^{\pi \sqrt{\frac{2}{\epsilon^\prime}}} \right),
\]
which is due to Friedman \cite{Friedman} and Nilli \cite{Nilli}. If $z = (2 - \epsilon) \sqrt{k - 1}$ and we instead fix $\epsilon$, Theorem \ref{thm:IntroMachine} gives the following.

\begin{Corollary}\label{cor:EpsilonCor}
Fix $\epsilon > 0$ and let $m$ be the minimal integer such that
\[
2 - \epsilon < \alpha_m = 2 \cos(\pi / m + 1).
\]
Let $v_\epsilon(k)$ be the maximal number of vertices of a $k$-regular graph $X$ with $\mu_1(X) \le (2 - \epsilon) \sqrt{k - 1}$. Then
\[
v_\epsilon(k) = \mathrm{O} \left( \left( \frac{k}{\sqrt{k - 1}} \right)^{2 m - 1} \right).
\]
In other words, the constant $C(k, \epsilon)$ from Theorem \ref{thm:QABT} satisfies
\[
C(k, \epsilon) = \Omega \left( \left( \frac{k}{\sqrt{k - 1}} \right)^{1 - 2 m} \right).
\]
\end{Corollary}

It would be interesting to better understand how $v(k, z)$ changes with both $k$ and $z$. The remainder of the paper explores some small values and compares the bounds one can extract by our methods with known results for small $k$ and $\mu_1$. For example, the complete graph on $k$ vertices $K_k$, which is $(k - 1)$-regular, has $\mu_1(K_k) = -1$ for all $k$. The complete bipartite graph $K_{k, k}$, which is $k$-regular, has $\mu_1(X) = 0$ for all $k$. In particular, $v(k, z) \ge 2 k$ for all $z \ge 0$ (in fact, one can prove that this is an equality). An easy application of our bounds shows that $v(k, z) \le 2 k + 2$, so our methods give the correct asymptotic behavior. 

We note that one can also derive a completely explicit bound from the proof of Theorem 1 in \cite{Cioaba}. As discussed there, this bound for $C(\epsilon, k)$, which is of order $(1/2)^{\mathrm{O}(\sqrt{k} \log(\sqrt{k}/\epsilon)/\epsilon)}$, is not as strong asymptotically as those of Friedman \cite{Friedman} and Nilli \cite{Nilli}. Our bounds are also more effective, even in a practical sense. For example, our methods show that a $3$-regular graph $X$ with $\mu_1(X) \le 1$ has at most $20$ vertices (see \S \ref{ssec:3Reg}), but the methods from \cite{Cioaba} only give a bound of $473$. Expanding further, our methods can prove the following, which was known previously.

\begin{Theorem}\label{thm:3Bounds}
Let $X$ be a connected $3$-regular graph.
\begin{enumerate}


\item If $\mu_1(X) \le 1$, then $X$ is one of the following six graphs: the complete graph on $4$ vertices $K_4$, the complete bipartite graph of type $(3,3)$ $K_{3, 3}$, the triangular prism graph $Y_2$, the $3$-dimensional cube $C$, the Wagner graph $W$, and the Petersen graph $P$.

\item There are exactly four $3$-regular graphs with $\mu_1(X) = 1$: $K_{3, 3}$, $C$, $W$, and $P$.

\item The prism graph $Y_2$ is the only $3$-regular graph with $\mu_1(X) = 0$.

\item The complete graph on $4$ vertices is the unique $3$-regular graph with $\mu_1(X) = -1$.

\end{enumerate}
\end{Theorem}

Uniqueness of $K_4$ amongst $3$-regular graphs with $\mu_1 = -1$ is easy to prove without our methods, and a similar statement holds for all $k$. The Wagner graph is also known as a M\"obius ladder graph and is a circulant graph, being the Cayley graph of $\mathbb{Z} / 8 \mathbb{Z}$ with generators $\{\pm1, 4\}$. Appendix I contains more on the graphs in (2)-(5) of Theorem \ref{thm:3Bounds}. We wrote a program, freely available from the third author's website, that allows one to recreate our results or calculate bounds for any $k$. Rather than including large tables of bounds, we include several small tables and will make this program widely available so interested readers can compute bounds not included in this paper.

We close with some remarks on literature that appeared since this project was completed. Koledin and Stani\'c \cite{KS} published a classification of graphs with $\mu_1 \le 1$ (for important early work on this, see \cite{CGSS}). Our methods cannot give a complete classification. One can compare the tables at the end of this paper to their results, proved by completely different methods, to see the behavior of our bounds in comparison to reality in this simple case. Finally, while we were finalizing a new version of this paper, Cioab\v a--Koolen--Nozaki--Vermette sent us a preprint \cite{CKNV} that answers many questions that arise in this paper and gives interesting new information about $v(k, z)$; we refer the reader there for statements and other comments on the literature.





\subsubsection*{Acknowledgments} We thank Stephen Debacker for his help in getting this project going and Sebastian Cioab\v a for communication related to \cite{CKNV}.

\section{The Quantitative Alon--Boppana Theorem}\label{sec:ProofOfABT}

In this section, we sketch the proof of Theorem \ref{thm:QABT} and explain our method for optimizing the constant $C(k, \epsilon)$. Our exposition is based on the treatment given in the book of Davidoff, Sarnak, and Valette \cite{DSV}, and we refer the reader there for complete details.

For any nonnegative integer $m$, let $U_m(x)$ be the $m^{th}$ Chebyshev polynomial of second kind. This is the polynomial of degree $m$ such that
\[
U_m(\cos \theta) = \frac{\sin (m + 1) \theta}{\sin \theta}
\]
for all $\theta \in \R$. From the trace formula for finite $k$-regular graphs \cite[Theorem 1.4.6]{DSV}, we have the following.

\begin{Theorem}\label{thm:TraceFormula}
Let $X$ be a finite connected $k$-regular graph with $n$ vertices and $\{\mu_j(X)\}$ be the eigenvalues of its adjacency matrix. Then for all nonnegative integers $m$,
\begin{equation}\label{eq:TraceFormula}
(k - 1)^{\frac{m}{2}} \sum_{j = 0}^{n - 1} U_m \left( \frac{\mu_j(X)}{2 \sqrt{k - 1}} \right) \ge 0.
\end{equation}
\end{Theorem}

It is convenient to define $V_m(x) = U_m(\frac{x}{2})$ (note that \cite{DSV} uses $X_m$). We then have the following.

\begin{Proposition}\label{prop:Measures}
Choose any $\epsilon > 0$ and $L \ge 2$. There exists a positive real constant $C(L, \epsilon) > 0$ such that for any probability measure $\nu$ on $[-L, L]$ with
\[
\int_{-L}^L V_m(x) \dnu \ge 0
\]
for every nonnegative integer $m$, we must have
\begin{equation}\label{eq:MeasureInequality}
\nu \left( [2 - \epsilon, L] \right) \ge C(L, \epsilon).
\end{equation}
\end{Proposition}

Before giving a sketch of the proof of Proposition \ref{prop:Measures}, we use it to give the proof of Theorem \ref{thm:QABT}.

\begin{proof}[Proof of Quantitative Alon--Boppana]
Let $X$ be a connected $k$-regular graph with $n$ vertices. Choose any $\epsilon > 0$, and set
\[
L = \frac{k}{\sqrt{k - 1}} \ge 2.
\]
Let
\[
\nu = \frac{1}{n} \bigoplus_{j = 0}^{n - 1} \delta \left( \frac{\mu_j(X)}{\sqrt{k - 1}} \right),
\]
where $\delta(x)$ is the Dirac measure at $x$. Then $\nu$ is a probability measure, and for every integer $m \ge 0$ Theorem \ref{thm:TraceFormula} gives
\[
\int_{-L}^L V_m(x) \dnu = \sum_{j = 0}^{n - 1} U_m \left( \frac{\mu_j(X)}{2 \sqrt{k - 1}} \right) \ge 0.
\]
By Proposition \ref{prop:Measures}, there is a constant $C(L, \epsilon) > 0$ such that
\[
\nu \left( [2 - \epsilon, L] \right) \ge C(L, \epsilon).
\]
By definition of $\nu$ as a weighted sum of Dirac measures,
\[
\nu \left( [2 - \epsilon, L] \right) = \frac{1}{n} \left( \# \left\{ j\ \mid\ \frac{\mu_j(X)}{\sqrt{k - 1}} \in [2 - \epsilon, L] \right\} \right) =
\]
\[
\frac{1}{n} \left( \# \left\{ j\ \mid\ \mu_j(X) \in [(2 - \epsilon)\sqrt{k - 1}, k] \right\} \right).
\]
Taking $C(k, \epsilon) = C(L, \epsilon)$ proves the theorem.
\end{proof}

Since some elements of the proof will be important, we now sketch the proof of Proposition \ref{prop:Measures}.

\begin{proof}[Sketch of proof for Proposition \ref{prop:Measures}]
First, note that the roots of $V_m(x)$ are precisely $2 \cos(\ell \pi / m + 1)$ for $\ell \in \{1, \dots, m\}$. Let $\alpha_m = 2 \cos(\pi / m + 1)$ be the largest root of $V_m$. One then uses the recursion formula for $V_m$ to show that
\[
Y_m(x) := \frac{V_m(x)^2}{x - \alpha_m} = \sum_{i = 0}^{2 m - 1} y_{m, i} V_i(x)
\]
where $y_{m, i} \ge 0$ for each $i \in \{0, \dots, 2 m - 1\}$. Notice that we have:
\[
\begin{cases} Y_m(x) < 0 & x < \alpha_m \\ Y_m(x) > 0 & x > \alpha_m \end{cases}
\]

Now, suppose that $L \ge 2$ and $\nu$ is a probability measure on $[-L, L]$ satisfying the conditions of the proposition such that $\nu([2 - \epsilon, L]) = 0$, i.e., $\nu$ is supported on $[-L, 2 - \epsilon]$. We can choose $m$ large enough that $\alpha_m > 2 - \epsilon$, which implies that $Y_m(x) \le 0$ for every $x$ in the support of $\nu$ and so
\[
\int_{-L}^L Y_m(x) \dnu \le 0.
\]
However,
\[
\int_{-L}^L Y_m(x) \dnu = \sum_{i = 0}^{2 m - 1} y_{m, i} \int_{-L}^L V_m(x) \dnu \ge 0
\]
by our assumption on $\nu$. It follows that the support of $\nu$ is a subset of the roots of $V_m$. The same conclusion must hold for any $m^\prime > m$ by the same argument. However, choosing any $m^\prime > m$ such that the roots of $V_m$ and $V_{m^\prime}$ are disjoint, we see that the support of $\nu$ is empty. This is a contradiction. Therefore, $\nu([2 - \epsilon, L]) > 0$ for any $\nu$ satisfying the conditions of the proposition. The existence of the constant $C(L, \epsilon)$ follows from a compactness argument in the space of measures satisfying the conditions of the proposition.
\end{proof}

We now describe our strategy for finding effective bounds for $C(k, \epsilon)$. Let $X$ be a connected $k$-regular graph with $n$ vertices, set $L = k / \sqrt{k - 1}$, choose $\epsilon > 0$, and let $\nu$ be a probability measure on $\mathcal{I} = [-L, L]$. If $f : \mathcal{I} \to \R$ is a $\nu$-measurable function and $Y \subset \mathcal{I}$ a $\nu$-measurable subset, then
\begin{equation}\label{eq:MeasureBound}
\nu(Y) \inf_{y \in Y} f(y) \le \int_Y f(y) \mathrm{d}\nu(y) \le \nu(Y) \sup_{y \in Y} f(y).
\end{equation}

Now suppose that $z_0 \in \mathcal{I}$, $\mathcal{I}_1 = [-L, z_0]$, $\mathcal{I}_2 = [z_0, L]$, and $f : \mathcal{I} \to \R$ is a $\nu$-measurable function such that $\int_{\mathcal{I}} f(x) \dnu \ge 0$. Moreover, suppose that $f$ is negative on $\mathcal{I}_1$. Then, as noted in the proof of Proposition \ref{prop:Measures}, $\nu(\mathcal{I}_2) > 0$. Define
\[
M_j = \sup_{y \in \mathcal{I}_j} f(y)
\]
Then $M_1 < 0 < M_2$, and we have the following string of implications:
\begin{eqnarray}
- \int_{\mathcal{I}_1} f(y) \mathrm{d}\nu(y) &\le& \int_{\mathcal{I}_2} f(y) \mathrm{d}\nu(y) \nonumber \\
-M_1 \nu(\mathcal{I}_1) &\le& M_2 \nu(\mathcal{I}_2) \label{eq:MaxBound} \\
-M_1 &\le& \nu(\mathcal{I}_2)(M_2 - M_1) \nonumber \\
\frac{-M_1}{M_2 - M_1} &\le& \nu(\mathcal{I}_2) \label{eq:fMaxBound0}.
\end{eqnarray}
Thus we obtain a positive lower bound for $\nu(\mathcal{I}_2)$. In \S\ref{sec:Asymptotic}, we implement this simple idea using certain linear combinations of the functions $V_m(x)$ defined above to find effective lower bounds for the constant $C(k, \epsilon)$ in the Theorem \ref{thm:QABT}, i.e., upper bounds for $v(k, (2 - \epsilon)\sqrt{k - 1})$.

\section{Behavior for arbitrary $\mu_1$ and $k$}\label{sec:Asymptotic}

Fix a real number $z$. For $k$ sufficiently large, Theorem \ref{thm:Finiteness} states that the number of $k$-regular graphs $X$ with $\mu_1(X) \le z$ is finite. In this section we consider the behavior of our methods for bounding $v(k, z)$. That is, we study the growth, in terms of $z$ and $k$, of the maximum number of vertices of a $k$-regular graph $X$ with $\mu_1(X) \le z$. 

Fix $k \ge 3$, define $L = k / \sqrt{k - 1}$, and set $\mathcal{I} = [-L, L]$. For any nonnegative integer $m$, let $U_m(x)$ be the $m^{th}$ Chebyshev polynomial of $2^{nd}$ kind and $V_m(x) = U_m(\frac{x}{2})$. Suppose that $\nu$ is any probability measure on $\mathcal{I}$ so that
\[
\int_{\mathcal{I}} V_m(x) \dnu \ge 0
\]
for all $m$. Choose any $w \in \left( -k, 2\sqrt{k - 1} \right)$ and let $\epsilon > 0$ be the number such that $w = (2 - \epsilon)\sqrt{k - 1}$. Set $z = 2 - \epsilon = w / \sqrt{k - 1}$. Our goal is to give a lower bound for the constant $C(k, \epsilon)$ such that
\[
\nu([z, L]) \ge C(k, \epsilon).
\]

For any $N \ge 0$, choose $\alpha_0, \dots, \alpha_N > 0$ and define
\[
f(x) = \sum_{m = 0}^N \alpha_m V_m(x).
\]
Then $\int_{\mathcal{I}} f(x) \dnu \ge 0$. Furthermore, we suppose:
\begin{itemize}

\item[$(\star)$] $f$ is strictly negative on $\mathcal{I}_1 = [-L, z]$.

\end{itemize}
Set $\mathcal{I}_2 = [z, L]$ and define
\[
M_j = \sup_{y \in \mathcal{I}_j} f(y) \nonumber.
\]
It is not hard to see that $M_2 = f(L)$ and $M_2 > 0 > M_1$. By \eqref{eq:fMaxBound0} in \S\ref{sec:ProofOfABT}, we then have
\begin{equation}\label{eq:fMaxBound}
\nu([z, L]) \ge \frac{-M_1}{M_2 - M_1} > 0.
\end{equation}
Using this equation, we can now let $C(k, \epsilon)$ be a real number such that
\begin{equation}\label{eq:cBound}
C(k, \epsilon) \geq \frac{-M_1}{M_2 - M_1}.
\end{equation}
Applying Theorem~\ref{thm:Finiteness} gives that
\begin{equation}\label{eq:vBound}
v(k,(2-\epsilon)\sqrt{k-1}) \leq \frac{M_2 - M_1}{-M_1}.
\end{equation}

Note that \eqref{eq:fMaxBound} is invariant under scaling $f$. Therefore, we make the normalization:
\begin{itemize}

\item[$(\star\star)$] $\sum \alpha_m^2 = 1$.

\end{itemize}
For the remainder of this section $f(x) = \sum \alpha_m V_m(x)$ will be a linear combination with nonnegative coefficients such that $(\star)$ and $(\star \star)$ hold.


To exhibit how one can explicitly apply these bounds, we first consider the function $f(x) = V_1(x) = x$. On the interval
\[
\left[ -\frac{k}{\sqrt{k - 1}}, \frac{z}{\sqrt{k - 1}} \right],
\]
the maximum value of $f$ equals $z / \sqrt{k - 1}$. On the interval
\[
\left[ \frac{z}{\sqrt{k - 1}}, \frac{k}{\sqrt{k - 1}} \right]
\]
the maximum value is $k / \sqrt{k - 1}$.

Inserting this into \eqref{eq:fMaxBound} and canceling the factors of $\sqrt{k - 1}$ shows that for any $z < 0$ and natural number $k$, a $k$-regular graph $X$ with $\mu_1(X) \le 0$ has at most
\[
\frac{z - k}{z}
\]
vertices. Taking $z = -1$, we have the following, which one can also prove by more elementary means using the fact that the trace of the adjacency matrix is $0$.

\begin{Corollary}\label{cor:z-1}
If $X$ is a $k$-regular graph with $\mu_1(X) \le -1$, then $X$ is the complete graph of $k + 1$ vertices.
\end{Corollary}

We now consider the case of two terms:
\[
f_\sigma(x) = V_1(x) + \sigma V_2(x) = \sigma x^2 + x - \sigma.
\]
Note our different normalization of the coefficients than from \S \ref{sec:Asymptotic}.
\begin{figure}[t]
\centerline{\includegraphics[width=8cm]{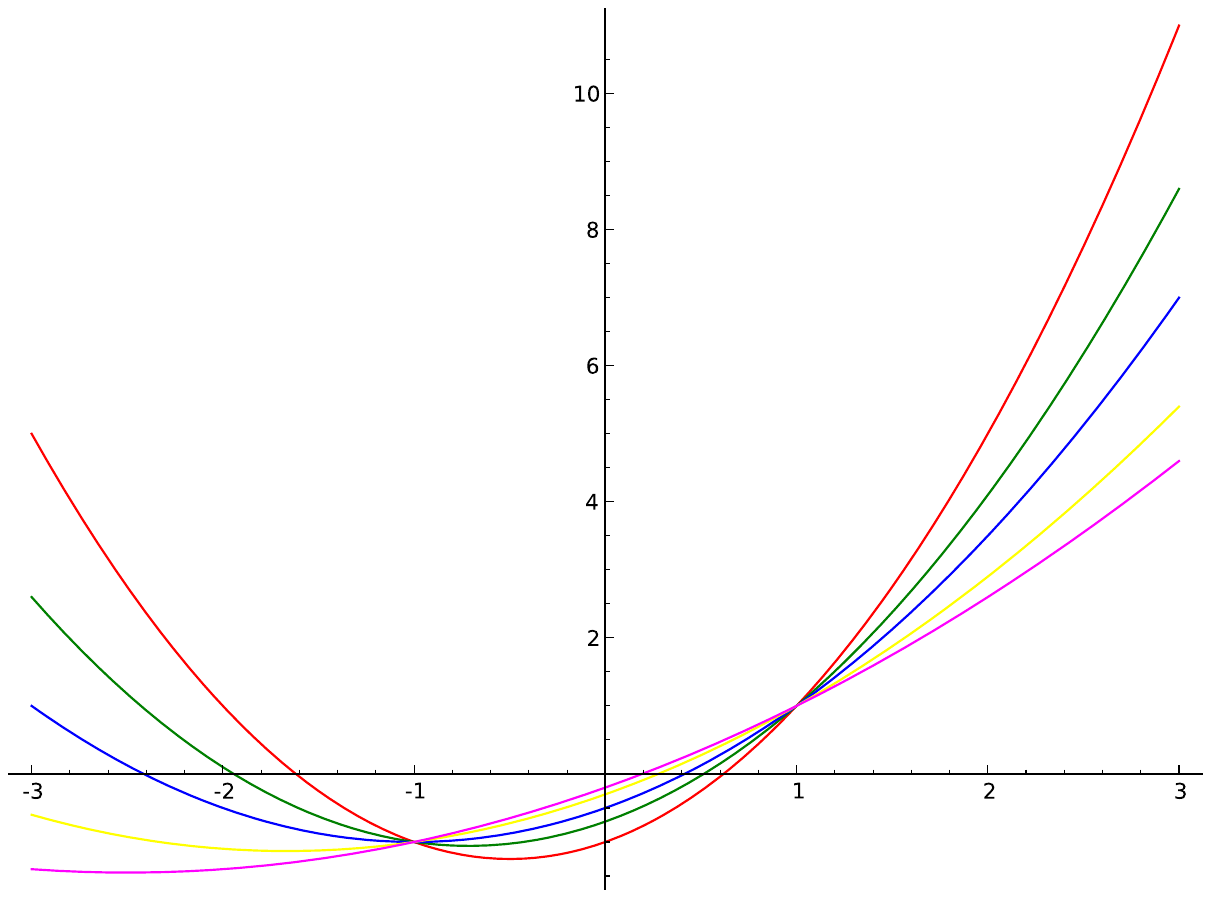}}
\caption{The functions $f_\sigma(x)$ for $\sigma \in (0, 1)$.}\label{fig:Degree2}
\end{figure}
For fixed $z$ as above, we need $f_\sigma$ to be strictly negative on the interval
\[
\mathcal{I}_1 = [-k / \sqrt{k - 1}, z / \sqrt{k - 1}].
\]
This happens if and only if
\begin{equation}\label{eq:SigmaBound}
\frac{z \sqrt{k - 1}}{k - 1 - z^2} < \sigma < \frac{k \sqrt{k - 1}}{k^2 - k + 1}.
\end{equation}

Now suppose that $z < 1$. To obtain a bound for $c_z(k)$ from $f_\sigma$, we note first that the maximum of $f_\sigma$ on
\[
\mathcal{I}_2 = [z / \sqrt{k - 1}, k / \sqrt{k - 1}]
\]
is
\[
M_2 = f \left( \frac{k}{\sqrt{k - 1}} \right) = \sigma \frac{k^2}{k - 1} + \frac{k}{\sqrt{k - 1}} - \sigma.
\]
For any $z < (k - 1)/k$, we can choose
\begin{equation}\label{eq:SigmaChoice}
\sigma = \frac{\sqrt{k-1}}{k - z},
\end{equation}
and simple analysis as in the linear case gives a bound for $c_z(k)$ that is roughly linear in $k$.

At $z = 0$, things are especially nice. Taking $\sigma = \sqrt{k - 1}/k$ as above, the values of $f_\sigma$ and the endpoints $-k / \sqrt{k - 1}$ and $0$ of $\mathcal{I}_2$ give
\[
M_1 = -\sqrt{k - 1}/k.
\]
The maximum value of $f_\sigma$ on $\mathcal{I}_2$ is its value at the endpoint, so
\[
M_2 = \frac{2 k}{\sqrt{k - 1}} - \frac{\sqrt{k - 1}}{k}.
\]
Thus
\[
v(k, 0) \le \frac{M_2 - M_1}{-M_1} = \frac{2 k^2}{k - 1} \le 2 k + 3.
\]
In fact, this is a strict inequality for $k \ge 4$, so one can in fact deduce that $v(k, 0) \le 2 k + 2$. It is known that $v(k, 0) = 2 k$, so our methods give the correct asymptotic growth.



When $z = 1$, the above methods break down. In other words, one must use additional Chebyshev polynomials in order to find effective bounds. We now describe a process by which one can calculate very good bounds for any $z < 2 \sqrt{k - 1}$. For any $m > 0$, consider the function
\begin{equation}\label{eq:FmDefinition}
F_m(x) = \sum_{j = 0}^m V_{2 j}(x).
\end{equation}
We first note that this function satisfies the following important properties.

\begin{Proposition}\label{prop:FmProperties}
Let $F_m(x)$, $m \ge 2$, be the function defined in \eqref{eq:FmDefinition}. For every $1 \le k \le m$, $\cos( k \pi / m + 1)$ is a double root of $F_m$ and $F_m(x) \ge 0$ for all $x \in \R$. Set $\alpha_m = 2 \cos(\pi / m + 1)$. Then
\begin{equation}\label{eq:FmhatDefinition}
\widehat{F}_m(x) = \frac{F_m(x)}{x - \alpha_m}
\end{equation}
is nonpositive for $x < \alpha_m$ and strictly positive for $x > \alpha_m$. Moreover, one has
\[
\widehat{F}_m(x) = \sum_{j = 0}^{2m - 1} y_{m, j} V_j(x)
\]
with $y_{m, j} \ge 0$ for all $j$.
\end{Proposition}

\begin{proof}
That $\cos(k \pi / m + 1)$ is a double root of $F_m$ for every $1 \le k \le m$ follows from elementary manipulations of Chebyshev polynomials evaluated at cosines.
It follows from basic calculus that $F_m(x) \ge 0$ for all $x \in \R$ and that $\widehat{F}_m(x)$ is positive for $x > \alpha_m$ and nonnegative for $x < \alpha_m$.

It remains to prove the last assertion, namely that $\widehat{F}_m$ is a linear combination of the functions $V_j(x)$ with nonnegative coefficients. We first note that for any coefficients $c_j$, 
\begin{eqnarray}
(x - \alpha_m) \sum_{j = 0}^{2 m - 1} c_j V_j(x) &=& \nonumber \\
(c_1 - \alpha_m c_0) V_0(x) + \sum_{j = 1}^{2 m - 2} (c_{j + 1} + c_{j - 1} - \alpha_m c_j) V_j(x) &+& \label{eq:FhatCoeffRelations} \\
(c_{2 m - 2} - \alpha_m c_{2 m - 1}) V_{2 m - 1}(x) + c_{2 m - 1} V_{2 m}(x) && \nonumber
\end{eqnarray}
(cf.~Proposition 1.4.8 in \cite{DSV}). For \eqref{eq:FhatCoeffRelations} to equal $F_m(x)$, we therefore need:
\begin{eqnarray}
c_1 - \alpha_m c_0 &=& 1 \nonumber \\
c_{j + 1} + c_{j - 1} - \alpha_m c_j &=& \left\{ \begin{matrix} 0 & 1 \le j \le 2 m - 3\ \textrm{odd} \\ 1 & 2 \le j \le 2 m - 2\ \textrm{even}\end{matrix} \right. \nonumber \\
c_{2 m - 2} - \alpha_m c_{2 m - 1} &=& 0 \nonumber \\
c_{2 m - 1} &=& 1 \nonumber
\end{eqnarray}
One can easily check using the standard relations for Chebyshev polynomials that we can take:
\begin{eqnarray}
c_{2 k} &=& \sum_{i = 0}^{m - 1 - k} V_{2 i + 1}(\alpha_m) \label{eq:Coeff1} \\
c_{2 k + 1} &=& \sum_{i = 0}^{m - 1 - k} V_{2 i}(\alpha_m) \label{eq:Coeff2}
\end{eqnarray}
Then $V_j(\alpha_m) > 0$ for every $j < 2 m$, so every $c_j$ is positive. This proves the proposition.
\end{proof}

Figure \ref{fig:Fms} shows graphs of $F_m(x)$ for small $m$.
\begin{figure}[t]
\centerline{\includegraphics[width=8cm]{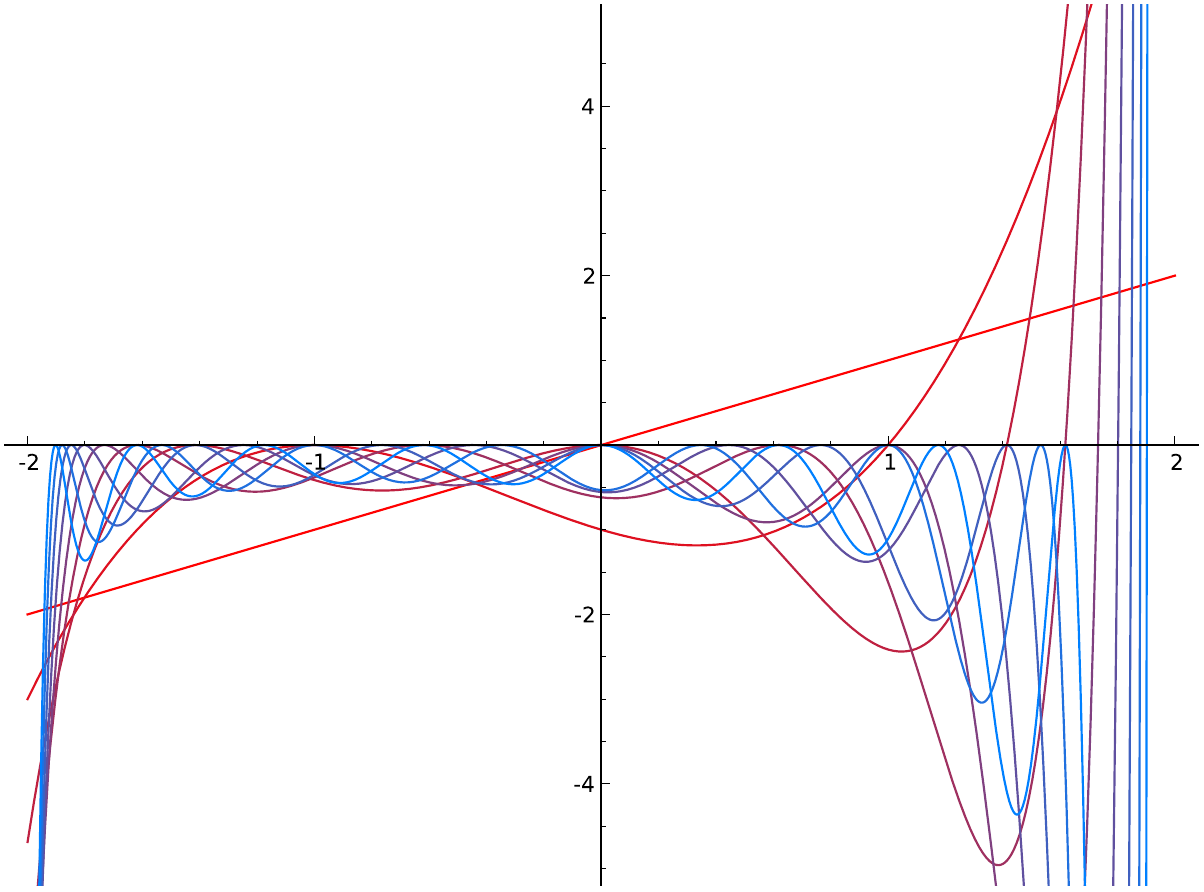}}
\caption{The functions $\widehat{F}_m(x)$ for small $m$.}\label{fig:Fms}
\end{figure}
Unfortunately, while it is a linear combination of Chebyshev polynomials with positive coefficients, $\widehat{F}_m(x)$ is not suitable for constructing the bounds under consideration in this paper because it is not strictly negative for $x < \alpha_m$. Therefore, for any $\mu_1$ such that $\mu_1 / \sqrt{k - 1} < \alpha_m$, the bound \eqref{eq:fMaxBound} from \S \ref{sec:Asymptotic} is always zero. To explain our strategy for extracting bounds from $\widehat{F}_m$, we begin with the following lemma, which still follows the strategy of \cite{DSV}.

\begin{Lemma}\label{lem:DownShift}
Let $\mathcal{I} = \mathcal{I}_1 \cup \mathcal{I}_2$ be an interval where $\mathcal{I}_1$ and $\mathcal{I}_2$ are intervals with disjoint interiors. Let $\nu$ be a probability measure on $\mathcal{I}$ such that $\int_{\mathcal{I}} V_j(x) \dnu \ge 0$ for all $j \ge 0$. Suppose that
\[
f(x) = \sum_{j = 0}^n c_j V_j(x)
\]
with $c_j \ge 0$ for all $0 \le j \le n$, and that $f(x) \le 0$ for all $x \in \mathcal{I}_1$. Then
\[
\nu(\mathcal{I}_2) \ge \frac{c_0 - M_1}{M_2 - M_1} \quad \left( M_j = \max_{x \in \mathcal{I}_,} f(x) \right).
\]
\end{Lemma}

\begin{proof}
Applying \eqref{eq:fMaxBound} to
\[
\widetilde{f}(x) = f(x) - c_0 = \sum_{j = 1}^n c_j V_j(x),
\]
which is still a linear combination of $V_j$s with positive coefficients, we get
\[
\nu(\mathcal{I}_2) \ge \frac{-\widetilde{M}_1}{\widetilde{M}_2 - \widetilde{M}_1},
\]
where
\[
\widetilde{M}_j = \max_{x \in \mathcal{I}_j} \widetilde{f}(x).
\]
Since $\widetilde{M}_j = M_j - c_0$, the lemma follows.
\end{proof}

Unfortunately, Lemma \ref{lem:DownShift} does not improve our situation, since $c_0 = 0$ for $\widehat{F}_m(x)$. We also must shift $\widehat{F}_m$ by some positive $s$, which we can do by the following general proposition.

\begin{Proposition}\label{prop:LeftShift}
Let
\[
F(x) = \sum_{j = 0}^n c_j V_j(x)
\]
with $c_j \ge 0$ for all $j$. For any $s > 0$ and each $0 \le j \le n$, there is a polynomial $q_j = q_{F, s, j}$ such that $q_j(s) \ge 0$ for all $j$ and
\[
F(x + s) = \sum_{j = 0}^n q_j(s) V_j(x).
\]
That is, the function $x \mapsto F(x + s)$ remains a linear combination of Chebyshev polynomials with nonnegative coefficients. Moreover, $q_0(s) > 0$ for all $s > 0$.
\end{Proposition}

\begin{proof}
It suffices to prove the proposition for $V_j(x)$, $0 \le j \le \infty$. We proceed by induction, and leave checking the first couple cases to the reader. In particular, fix $s > 0$ and suppose that
\begin{equation}\label{eq:InductionShift}
V_k(x + s) = \sum_{i = 0}^k \epsilon_{k, i} V_i(x) \quad 0 \le k \le j - 1.
\end{equation}
Then
\begin{eqnarray}
V_j(x + s) = (x + s) V_{j - 1}(x + s) - V_{j - 2}(x + s) &=& \nonumber \\
(\epsilon_{j - 1, 1} + s \epsilon_{j - 1, 0} - \epsilon_{j - 2, 0}) V_0(x) &+& \nonumber \\
\sum_{i = 1}^{j - 2} (\epsilon_{j - 1, i + 1} + \epsilon_{j - 1, i - 1} + s \epsilon_{j - 1, i} - \epsilon_{j - 2, i}) V_i(x) &+& \label{eq:ShiftInduction} \\
(\epsilon_{j - 1, j - 2} + s \epsilon_{j - 1, j - 1}) V_{j - 1}(x) + \epsilon_{j - 1, j - 1} V_j(x). \nonumber &&
\end{eqnarray}
The inductive hypothesis then implies that $V_{j - 1}(x)$ and $V_j(x)$ have positive coefficients. In fact, by induction and the fact that $V_0(x) = 1$, we see that $\epsilon_{k, k} = 1$ for all $k$ (independent of $s$).

To prove the proposition, it suffices to show that
\[
\epsilon_{j - 1, i + 1} - \epsilon_{j - 2, i} \ge 0
\]
for every $0 \le i \le j - 2$. Yet again, we induct. The above calculation shows that
\[
\epsilon_{j - 1, i + 1} = \epsilon_{j - 2, i + 2} + \epsilon_{j - 2, i} + s \epsilon_{j - 2, i + 1} - \epsilon_{j - 3, i + 1},
\]
and the inductive hypothesis implies that $\epsilon_{j - 2, i + 2} - \epsilon_{j - 3, i + 1} \ge 0$ for $0 \le i \le j - 3$, so
\[
\epsilon_{j - 1, i + 1} - \epsilon_{j - 2, i} = \epsilon_{j - 2, i + 2} + s \epsilon_{j - 2, i + 1} - \epsilon_{j - 3, i + 1} \ge 0
\]
in those cases. It remains to consider the case $i = j - 2$, where
\[
\epsilon_{j - 1, i + 1} - \epsilon_{j - 2, i} = \epsilon_{j - 1, j - 1} - \epsilon_{j - 2, j - 2} = 1 - 1 = 0.
\]
Also note that the inductive definition for each $\epsilon$ implies that it is a polynomial in $s$.

In remains to show that $\epsilon_{j, 0} > 0$. To see this,
\[
\epsilon_{j, 0} = \epsilon_{j - 1, 1} + s \epsilon_{j - 1, 0} - \epsilon_{j - 2, 0}.
\]
One last induction assumes $\epsilon_{j - 1, 0} > 0$, and we saw above that
\[
\epsilon_{j - 1, 1} - \epsilon_{j - 2, 0} \ge 0.
\]
Since $s > 0$, the claim follows. This completes the proof of the proposition.
\end{proof}

The above leads us to the following technical result, which is the best-optimized function we found for computing vertex bounds.

\begin{Theorem}\label{thm:BoundMachine}
Fix $z \in \R$ and a positive integer $k$ large enough that $z < 2 \sqrt{k - 1}$. Let $m$ be any positive integer such that
\[
\frac{z}{\sqrt{k - 1}} < \alpha_m = 2 \cos(\pi / m + 1).
\]
Set $L = k / \sqrt{k - 1}$ and define
\[
\widehat{F}_m(x) = \sum_{j = 0}^m \frac{V_j(x)}{x - \alpha_m}.
\]
For any real number
\[
0 < s < \alpha_m - \frac{z}{\sqrt{k - 1}},
\]
write the function $x \mapsto \widehat{F}_m(x + s)$ as
\[
F_{m, s}(x) = \sum_{j = 0}^{2 m - 1} c_j(s) V_j(x),
\]
and define
\begin{equation}
M_{m, s} = \max_{\frac{z}{\sqrt{k - 1}} \le x \le L} F_{m, s}(x) = \widehat{F}_m(L + s). \label{eq:ShiftMax} \\
\end{equation}
If $X$ is any $k$-regular graph with $\mu_1(X) = z$, then $X$ has at most $\frac{M_{m, s}}{c_0(s)}$ vertices. That is,
\begin{equation}
v(k, z) \le \frac{M_{m, s}}{c_0(s)}. \label{eq:ShiftConstant}
\end{equation}
\end{Theorem}

Recall that $c_0(s) > 0$ by Proposition \ref{prop:LeftShift}. In order to make computations like those in \S \ref{sec:Asymptotic}, one is left, of course, with finding the optimal choice of $s$. We leave this optimization to the reader, who can use our code to perform such an optimization. As for the asymptotic bounds for the spectrum given by this method, the choice of $s$ is irrelevant.

Instead, we now fix $z \in \R$ and study the nature of our vertex bound \eqref{eq:ShiftConstant} and complete the proof of Theorem \ref{thm:IntroMachine}. Suppose that $k$ is sufficiently large that
\[
\frac{z}{\sqrt{k - 1}} < 2,
\]
so there are only finitely many $k$-regular graphs $X$ with $\mu_1(X) = z$. Then there is a minimal integer $m = m(k) > 0$ such that
\[
\frac{z}{\sqrt{k - 1}} < \alpha_m = 2 \cos(\pi / m + 1).
\]
That is,
\[
m = m_z(k) = \left\lceil \pi~\mathrm{arccos} \left( \frac{z}{2 \sqrt{k - 1}} \right)^{-1} \right \rceil - 1,
\]
where $\lceil~\rceil$ is the ceiling function\footnote{Actually, we take $m$ to be one larger when the expression inside the ceiling function is an integer, so there is a genuine gap between $z / \sqrt{k - 1}$ and $\alpha_m$.}. Fix an arbitrary real number $s$ such that
\[
0 < s < \alpha_m - \frac{z}{\sqrt{k - 1}},
\]
and consider the function $F_{m, s}(x)$ defined in Theorem \ref{thm:BoundMachine}. This is a polynomial of degree $2 m - 1$. For any such $s$, the quantity $M_{m, s}$ from \eqref{eq:ShiftMax} is $F_{m, s}(L) = \widehat{F}_m(L + s)$. Therefore the bound \eqref{eq:ShiftConstant} is precisely \eqref{eq:IntroCBound}, which proves Theorem \ref{thm:IntroMachine}. Now we prove Corollary \ref{cor:EpsilonCor}.

\begin{proof}[Proof of Corollary \ref{cor:EpsilonCor}]
Fix $\epsilon > 0$, and apply the above to $z = (2 - \epsilon) \sqrt{k - 1}$. Then
\[
m = \left\lceil \pi~\mathrm{arccos} (2 - \epsilon)^{-1} \right \rceil - 1
\]
is independent of $k$. Similarly, we need $s \in (0, \alpha_m - (2 - \epsilon))$. This interval is independent of $k$, so we can also fix $s$ independent of $k$. Thus the vertex bound
\[
\frac{\widehat{F}_m(L + s)}{c_0(s)}
\]
from Theorem \ref{thm:IntroMachine} is a function of $L = k / \sqrt{k - 1}$ of degree $2 m - 1$. The corollary follows.
\end{proof}


\subsection{Bounds for $3$-regular graphs}\label{ssec:3Reg}

Suppose that $k = 3$. To classify the $3$-regular graphs $X$ with $\mu_1(X) \le 1$, we must implement the above with $\epsilon = 2 - 1 / \sqrt{2}$. This is an excellent example of how one can export our methods to other settings, as it suffices to consider the first $3$ Chebyshev polynomials. More precisely, we consider:
\begin{eqnarray}
V_1(x) &=& x \nonumber \\
V_2(x) &=& x^2 - 1 \nonumber \\
V_3(x) &=& x^3 - 2 x \nonumber \\
f(x) &=& \alpha_1 V_1(x) + \alpha_2 V_2(x) + \alpha_3 V_3(x) \nonumber
\end{eqnarray}
with $\alpha_m \ge 0$ for each $m$ and $\sum \alpha_m^2 = 1$. In addition, we require that $f$ be strictly negative on the closed interval $\mathcal{I}_1 = [-3 / \sqrt{2}, 1 / \sqrt{2}]$.

One can use Python\footnote{Python code allowing one to implement the computations in this paper is available from the third author's website.} to optimize the choice of $\{\alpha_1, \alpha_2, \alpha_3\}$ and, using \eqref{eq:cBound}, prove that
\[
C(3, 2 - 1 / \sqrt{2}) > \frac{1}{24}.
\]
Extending the above process to 5 terms, we get that
\[
C(3, 2 - 1 / \sqrt{2}) > \frac{1}{21}.
\]
Using \eqref{eq:vBound} and the fact that $v(3, 1)\in\mathbb{N}$, this implies that $v(3, 1) \leq 20$. Since all such graphs are known, one can check (2)-(5) in Theorem \ref{thm:3Bounds} by brute force. Using similar analysis with six terms allows one to prove that $\mu_1(X) \le 2$ implies that $X$ has at most $105$ vertices. Unfortunately, it is not currently feasible to compute all $3$-regular graphs with at most $105$ vertices, so we cannot give a complete classification of the $3$-regular graphs with $\mu_1(X) \le 2$ by our methods. The largest $3$-regular graph we know with $\mu_1(X) = 2$ is the Levi graph, which has $30$ vertices. A referee communicated a combinatorial proof that indeed $v(3,2) = 30$, and this is also shown in \cite{CKNV}. However, we conjecture\footnote{This conjecture is proved in \cite{CKNV}.} that if $X$ is a $3$-regular graph with $\mu_1(X) \le 1.9$, then $X$ has at most $18$ vertices, in which case one can easily compute all such graphs.

\subsection{Bounds for $k$-regular graphs, $k \ge 4$}\label{ssec:4Reg}

Using the same analysis described in \S\ref{ssec:3Reg}, we consider the behavior of our bounds for $k$-regular graphs, $k \ge 4$. It appears that the rolling cube graph, which has $24$ vertices, is the largest $4$-regular graph with $\mu_1 \le 2$ and that the Doyle graph, which has $27$ vertices, is the largest with $\mu_1 \le 3$.\footnote{Since this paper was completed, \cite{CKNV} showed that in fact $v(4, 2) = 35$ and $v(4, 3) = 728$.}

\begin{table}
\begin{center}
\begin{tabular}{|c|c|}
\hline
$\mu_1$ upper bound & vertex upper bound \\
\hline
$-1$ & $5$ \\
\hline
$0$ & $11$ \\
\hline
$1$ & $23$ \\
\hline
$2$ & $77$ \\
\hline
\end{tabular}
\end{center}
\caption{Vertex bounds for $4$-regular graphs with small $\mu_1$}
\end{table}

\begin{table}
\begin{center}
\begin{tabular}{|c|c|}
\hline
$\mu_1$ upper bound & vertex upper bound \\
\hline
$-1$ & $6$ \\
\hline
$0$ & $12$ \\
\hline
$1$ & $23$ \\
\hline
\end{tabular}
\end{center}
\caption{Vertex bounds for $5$-regular graphs with small $\mu_1$}
\end{table}

\begin{table}
\begin{center}
\begin{tabular}{|c|c|}
\hline
$\mu_1$ upper bound & vertex upper bound \\
\hline
$-1$ & $7$ \\
\hline
$0$ & $14$ \\
\hline
$1$ & $25$ \\
\hline
$2$ & $115$ \\
\hline
\end{tabular}
\end{center}
\caption{Vertex bounds for $6$-regular graphs with small $\mu_1$}
\end{table}

\begin{table}
\begin{center}
\begin{tabular}{|c|c|}
\hline
$\mu_1$ upper bound & vertex upper bound \\
\hline
$-1$ & $8$ \\
\hline
$0$ & $16$ \\
\hline
$1$ & $27$ \\
\hline
$2$ & $80$ \\
\hline
\end{tabular}
\end{center}
\caption{Vertex bounds for $7$-regular graphs with small $\mu_1$}
\end{table}

\begin{table}
\begin{center}
\begin{tabular}{|c|c|}
\hline
$\mu_1$ upper bound & vertex upper bound \\
\hline
$-1$ & $9$ \\
\hline
$0$ & $18$ \\
\hline
$1$ & $30$ \\
\hline
$2$ & $72$ \\
\hline
\end{tabular}
\end{center}
\caption{Vertex bounds for $8$-regular graphs with small $\mu_1$}
\end{table}

\begin{table}
\begin{center}
\begin{tabular}{|c|c|}
\hline
$\mu_1$ upper bound & vertex upper bound \\
\hline
$-1$ & $10$ \\
\hline
$0$ & $20$ \\
\hline
$1$ & $33$ \\
\hline
$2$ & $70$ \\
\hline
\end{tabular}
\end{center}
\caption{Vertex bounds for $9$-regular graphs with small $\mu_1$}
\end{table}

\begin{table}
\begin{center}
\begin{tabular}{|c|c|}
\hline
$\mu_1$ upper bound & vertex upper bound \\
\hline
$-1$ & $11$ \\
\hline
$0$ & $22$ \\
\hline
$1$ & $36$ \\
\hline
$2$ & $70$ \\
\hline
\end{tabular}
\end{center}
\caption{Vertex bounds for $10$-regular graphs with small $\mu_1$}
\end{table}

\section*{Appendix: The $3$-regular graphs with $\mu_1(X) \le 1$}

Below are the six $3$-regular graphs $X$ with $\mu_1(X) \le 1$.

\subsection*{$K_4$: The complete graph on $4$ vertices}

\begin{center}
\begin{tikzpicture}
[every node/.style={circle, fill=black}]
\node (n1) at (1,2) {};
\node (n2) at (0,1) {};
\node (n3) at (2,1) {};
\node (n4) at (1,0) {};

\foreach \from/\to in {n1/n2, n1/n3, n1/n4, n2/n3, n2/n4, n3/n4} \draw (\from) -- (\to);

\end{tikzpicture}
\end{center}
\begin{center}
Spectrum: $\{3, -1, -1, -1\}$
\end{center}

\subsection*{$K_{3, 3}$: The complete bipartite graph of type $(3,3)$}

\begin{center}
\begin{tikzpicture}
[every node/.style={circle, fill=black}]
\node (n1) at (1,0) {};
\node (n2) at (2,0) {};
\node (n3) at (3,0) {};
\node (n4) at (1,1) {};
\node (n5) at (2,1) {};
\node (n6) at (3,1) {};

\foreach \from/\to in {n1/n4, n1/n5, n1/n6, n2/n4, n2/n5, n2/n6, n3/n4, n3/n5, n3/n6} \draw (\from) -- (\to);

\end{tikzpicture}
\end{center}
\begin{center}
Spectrum: $\{3, 0, 0, 0, 0, -3\}$
\end{center}

\subsection*{$Y_2$: The triangular prism}

\begin{center}
\begin{tikzpicture}
[every node/.style={circle, fill=black}]
\node (n1) at (0,0) {};
\node (n2) at (3,0) {};
\node (n3) at (3/2,3*1.73205081/2) {};
\node (n4) at (1,1/2) {};
\node (n5) at (2,1/2) {};
\node (n6) at (3/2,1.366025405) {};

\foreach \from/\to in {n1/n2, n1/n3, n1/n4, n2/n3, n2/n5, n3/n6, n4/n5, n4/n6, n5/n6} \draw (\from) -- (\to);

\end{tikzpicture}
\end{center}
\begin{center}
Spectrum: $\{3, 1, 0, 0, -2, -2\}$
\end{center}

\subsection*{$C$: The $3$-dimensional cube}

\begin{center}
\begin{tikzpicture}
[every node/.style={circle, fill=black}]
\node (n1) at (0,0) {};
\node (n2) at (3,0) {};
\node (n3) at (0,3) {};
\node (n4) at (3,3) {};
\node (n5) at (3/4,3/4) {};
\node (n6) at (9/4,3/4) {};
\node (n7) at (3/4,9/4) {};
\node (n8) at (9/4,9/4) {};

\foreach \from/\to in {n1/n2, n1/n3, n1/n5, n2/n4, n2/n6, n3/n4, n3/n7, n4/n8, n5/n6, n5/n7, n6/n8, n7/n8} \draw (\from) -- (\to);

\end{tikzpicture}
\end{center}
\begin{center}
Spectrum: $\{3, 1, 1, 1, -1, -1, -1, 3\}$
\end{center}

\subsection*{$W$: The Wagner graph}

\begin{center}
\begin{tikzpicture}
[every node/.style={circle, fill=black}]
\node (n1) at (2,0) {};
\node (n2) at (2*0.70710678118,2*0.70710678118) {};
\node (n3) at (0,2) {};
\node (n4) at (-2*0.70710678118,2*0.70710678118) {};
\node (n5) at (-2,0) {};
\node (n6) at (-2*0.70710678118,-2*0.70710678118) {};
\node (n7) at (0,-2) {};
\node (n8) at (2*0.70710678118,-2*0.70710678118) {};

\foreach \from/\to in {n1/n2, n1/n8, n1/n5, n2/n3, n2/n6, n3/n4, n3/n7, n4/n5, n4/n8, n5/n6, n6/n7, n7/n8} \draw (\from) -- (\to);

\end{tikzpicture}
\end{center}
\begin{center}
Spectrum: $\{3, 1, 1, -1+\sqrt{2}, -1+\sqrt{2}, -1, -1-\sqrt{2}, -1+\sqrt{2}\}$
\end{center}

\subsection*{$P$: The Petersen graph}

\begin{center}
\begin{tikzpicture}
[every node/.style={circle, fill=black}]
\node (n1) at (0,2) {};
\node (n2) at (2*-0.95105651629,2*0.30901699437) {};
\node (n3) at (-2*0.58778525229,-2*0.80901699437) {};
\node (n4) at (2*0.58778525229,-2*0.80901699437) {};
\node (n5) at (2*0.95105651629,2*0.30901699437) {};
\node (n6) at (0,1) {};
\node (n7) at (-0.95105651629,0.30901699437) {};
\node (n8) at (-0.58778525229,-0.80901699437) {};
\node (n9) at (0.58778525229,-0.80901699437) {};
\node (n10) at (0.95105651629,0.30901699437) {};

\foreach \from/\to in {n1/n2, n1/n5, n1/n6, n2/n3, n2/n7, n3/n4, n3/n8, n4/n5, n4/n9, n5/n10, n6/n8, n6/n9, n7/n9, n7/n10, n8/n10} \draw (\from) -- (\to);

\end{tikzpicture}
\end{center}
\begin{center}
Spectrum: $\{3, 1, 1, 1, 1, 1, -2, -2, -2, -2\}$
\end{center}

\bibliographystyle{plain}

\bibliography{GraphFiniteness}

\end{document}